\documentclass[a4paper,11pt]{article}

\title{Galois and universal universal coverings of linear categories and fibre products}

\author{Claude Cibils, Mar\'\i a Julia Redondo and Andrea Solotar
\thanks{This work has been supported by the projects  UBACYTX212, PIP-CONICET 112- 200801-00487, PICT-2007-02182 and MATHAMSUD-NOCOMALRET.
The second and third  authors are  research members of
CONICET (Argentina). The third author is a Regular Associate of ICTP Associate Scheme.}}

\date{}

\input{xy}
\xyoption{all}
\xyoption{poly}
\usepackage[all]{xy}
\usepackage{amsmath,amsthm,amsfonts,amssymb,stmaryrd}
\usepackage{latexsym}
\usepackage{color,palatino}
\usepackage{graphicx}%poner [dvips] para gr\'{a}ficos
\usepackage{float}  %Este packete permite controlar la flotacion de figuras al poner [H] o [h] como opcion de \begin{figure}
\RequirePackage{amssymb}
\RequirePackage[T1]{fontenc}

\usepackage[
  breaklinks,
  naturalnames,
  ]{hyperref}

\newtheorem{thm}{Theorem}[section]
\newtheorem{cor}[thm]{Corollary}
\newtheorem{lem}[thm]{Lemma}
\newtheorem{pro}[thm]{Proposition}
\newtheorem{defi}[thm]{Definition}%[section]
\newtheorem{rem}[thm]{Remark}%[section]
\newtheorem{exa}[thm]{Example}

\let\oldqed\qed
\renewcommand\qed{\oldqed\par\bigskip}

\def\B{{\mathcal B}}
\def\C{{\mathcal C}}
\def\D{{\mathcal D}}

\def\F{{\mathcal F}}

\def\U{{\mathcal U}}

\def\aut{{\mathrm {Aut}}}

\def\lim{\mathop{\rm lim}\nolimits}

\def\st{\mathsf{St}}
\def\mor{\mathsf{Mor}}

\def\pr{\mathsf{pr}}
\begin{document}

\maketitle

\begin{abstract}
Let $k$ be a commutative ring.
We study the behaviour of coverings of $k$-categories through 
fibre products and find a criterion for a covering to be Galois or universal.
\end{abstract}



%%%%%%%%%%%%%%%%%%%%%%%%%%%%%%%%%%%%%%%%%%%%%%%%%%%%%%%%%%%%%%%%%%%%%%%%%%%%%%%%%%%%%%%%%%%

\section{\sf Introduction}

In this paper we study Galois coverings of $k$-categories and we look for conditions that ensure the existence of a universal covering.
With this purpose in mind we consider the fibre product of two Galois coverings. Differences with the usual algebraic topology setting appear,
since the fibre product of coverings of $k$-categories does not provide in
general a covering through the projection functor. Nevertheless, the fibre product of a Galois covering with a fully faithful functor is a
Galois covering.

The theory of Galois coverings is deeply related to the fundamental group and gradings of $k$-linear categories.
In \cite{CRS} we introduce the intrinsic fundamental
group $\pi_1(\B)$ attached to a $k$-linear category $\B$ as the automorphism group of the fibre functor which sends any Galois covering $F$
to its fibre $F^{-1}(b)$ over a chosen object $b$.  This fundamental group takes into account the
linear structure of the category $\B$. It is intrinsic in the sense
that it is not attached to a presentation of the category by generators and relations. However, the fundamental groups previously
constructed by R.
Mart\'\i nez-Villa and J.A. de la Pe\~{n}a in \cite{MP} and by K. Bongartz and P. Gabriel \cite{boga,ga} associated to a presentation
of the category
by a quiver and relations and depending on it \cite{asde,buca} are quotients of the intrinsic $\pi_1$
we introduce.
In case a universal covering $U: \U\rightarrow\B$
exists, the fundamental group that we have defined is isomorphic to the
automorphism group $\aut (U)$.

The concept of Galois covering is also related to gradings and to the smash product construction.
In \cite{CRS2} we make an explicit comparison between Galois coverings and smash coverings of a $k$-category $\B$. More precisely,
we provide an equivalence between the category of Galois coverings of $\B$ and its full subcategory
whose objects are the smash product coverings. We consider the fundamental group that has been defined in
\cite{CRS} using Galois coverings and show that we can restrict to smash coverings when computing the fundamental
group $\pi_1(\B)$. From \cite{CM} we get that each connected grading of $\B$ provides a Galois covering through the
smash product, and hence the fundamental group is derived from all the groups grading the linear category in a connected way.
This allows the computation of the fundamental groups of matrix algebras, triangular algebras,
group algebras and diagonal algebras, see \cite{CRS2}.
In particular this provides a confirmation of the fact that the fundamental group of an algebra takes into account the matrix structure,
in other words it is not a Morita invariant: matrices do not admit a universal covering since there exist at
least two connected components of the category of connected gradings.

Our main results here are: a characterization of Galois coverings, since $F$ is Galois if and only if the fibre product of $F$ with itself
is a trivial covering; and a characterization of universal coverings, since $F$ is universal if and only if the fibre product of $F$
with any Galois covering
provides a trivial covering. We finish this article with an application.

\section{\sf Preliminaries}

In this section we recall some definitions and results from \cite{CRS} that will be useful for this article.

Let $k$ be a commutative ring. A $k$-category is a small category
$\B$ such that each morphism set ${}_y\B_x$ from an object
$x\in\B_0$ to an object $y\in\B_0$ is a $k$-module, the composition of
morphisms is $k$-bilinear and the identity at each object is central
in its endomorphism ring. In particular each endomorphism set is a
$k$-algebra, and ${}_y\B_x$ is a ${}_y\B_y - {}_x\B_x$-bimodule.

Each $k$-algebra $A$ provides a single object $k$-category $\B_A$ with
endomorphism ring $A$. Of course the structure of $A$ can be
described more precisely by choosing a finite set $E$ of orthogonal
idempotents of $A$, such that $\sum_{e\in A}e=1$, the
$k$-category $\B_{A,E}$ has set of objects $E$ and morphisms from
$e$ to $f$ the $k$-module $fAe$. Note that $\B_{A,\{1\}}=\B_A$.

\begin{defi} The star $\st_{b}\B$ of a $k$-category $\B$ at an
object $b$ is the direct sum of all the morphisms with source or target $b$ :
\[\st_{b}\B = \left(\bigoplus_{y\in\B_0} \ {}_y\B{_{b}}\right)\ \oplus
\ \left(\bigoplus_{y\in\B_0} \ {}_{b}\B{_y}\right).\] Note that
this $k$-module counts twice the endomorphism algebra at $b$.
\end{defi}

\begin{defi}\label{covering}
 Let $\C$ and $\B$ be $k$-categories. A $k$-functor
$F:\C\rightarrow\B$ is a \emph{covering} of $\B$ if it is surjective on
objects and if $F$ induces $k$-isomorphisms between the
corresponding stars. More precisely, for each $b\in\B_0$ and each
$x$ in the non-empty fibre $F^{-1}(b)$, the map
\[F_{b}^x:\st_x\C\longrightarrow\st_{b}\B.\]
induced by $F$ is a $k$-isomorphism.

\end{defi}

\begin{rem} Each star is the direct sum of the source star $\st^-_{b}\B =
\bigoplus_{y\in\B_0}{}_y\B{_{b}}$ and the target
star $\st^+_{b}\B=\bigoplus_{y\in\B_0} \
{}_{b}\B{_y}$. Since the  $\st^-$ and $\st^+$ components are stable under $F$,
the condition of the definition is equivalent to the
requirement that the corresponding target and source stars are
isomorphic through $F$.\\
Moreover this splitting goes further: the restriction of $F$ to
$\bigoplus_{y\in F^{-1}(c)} {}_y\C_x$
is $k$-isomorphic to the corresponding
$k$-module ${}_{c}\B_{b}$. The same holds
with respect to the target star and morphisms starting at all
objects in a single fibre.
\end{rem}

\begin{rem}
The previous facts show that Definition \ref{covering} coincides with
the one given
by K. Bongartz and P. Gabriel in \cite{boga}.
\end{rem}

\begin{defi}
Given $k$-categories $\B, \C, \D$, the set $\mor(F,G)$ from a covering
$F:\C\rightarrow\B$ to a
covering $G:\D\rightarrow\B$ is the set of
pairs of $k$-linear functors $(H,J)$
where $H: \C  \to \D$, $J: \B \to \B$ such that $GH=JF$, with $J$ an isomorphism. \\
We will consider within the group of automorphisms of a covering $F:\C \rightarrow \B$, the subgroup
$\aut_1 (F)$ of invertible endofunctors $H$ of $\C$ such that $FH=F$.
\end{defi}

We next recall the definition of a connected $k$-category. In fact,
given a covering, we shall be interested in each connected component.

We use the following notation: given a morphism $f$, its source object is denoted $s(f)$
and $t(f)$  is its target object. We will also make use of walks. For this purpose we consider the set of formal pairs $(f,\epsilon)$ as
``morphisms with sign'', where $f$ is a morphism in $\B$ and $\epsilon \in \{-1,1\}$.
We extend source and target maps to this set as follows:
\begin{center}
$s(f,1)=s(f)$, $s(f,-1)=t(f)$, $t(f,1)=t(f)$, $t(f,-1)=s(f)$.
\end{center}

\begin{defi}
Let $\B$ be a $k$-category. A non-zero \emph{walk} in $\B$ is a sequence of non-zero morphisms
with signs $(f_n,\epsilon_n) \dots (f_1,\epsilon_1)$ such that $s(f_{i+1}, \epsilon_{i+1})= t(f_{i}, \epsilon_{i})$.
We say that this walk goes from  $s(f_1, \epsilon_1)$ to $t(f_n, \epsilon_n)$. \\
A $k$-category $\B$ is \emph{connected} if any two objects $b$
and $c$ of $\B$ can be joined by a non-zero walk.
\end{defi}

\begin{pro}\label{connected} \cite[Proposition 2.8]{CRS}
Let $F:\C\longrightarrow\B$ be a covering of $k$-categories. If
$\C$ is connected, then $\B$ is connected.
\end{pro}

The following two statements correspond respectively to \cite[Proposition 2.9]{CRS} and \cite[Corollary 2.10]{CRS}.

\begin{pro}\cite{le}\label{equal}
Let $F:\C\longrightarrow\B$ and $G:\D\longrightarrow\B$ be coverings of $k$-linear
categories. Assume $\C$ is connected.
Two morphisms $(H_1,J)$, $(H_2,J)$ from $F$ to $G$ such that $H_1$ and $H_2$ coincide on some object are equal.
\end{pro}

\begin{cor}
Let $F: \C \to \B$ be a  connected covering of a $k$-linear category $\B$.  The group $\aut_1(F)=\{ (H,1): F \to F \mid H \ \mbox{an isomorphism of}\ \C\}$ acts freely on each fibre.
\end{cor}

Next we recall the definition of a Galois covering:

\begin{defi}
A covering $F: \C\longrightarrow\B$ of $k$-categories is a
\emph{Galois covering} if $\C$ is connected and if $\aut_1(F)$ acts
transitively on some fibre.
\end{defi}

Of course, the condition "on some fibre" actually means that the automorphism group acts
transitively at every fibre whenever it acts transitively on a
particular one, as it is proven in \cite{CRS, le, le2}.

\begin{pro} Let
$F:\C\longrightarrow\B$ be a Galois covering. Then $\aut_1 (F)$ acts
transitively on each fibre.
\end{pro}

The following structure theorem gives an explicit description of Galois coverings.

\begin{thm}\label{structure} \cite[Theorem 3.7]{CRS}.
Let  $F:\C\longrightarrow\B$ be a Galois covering. Then there exists a unique
isomorphism of categories $F':\C/\aut_1(F)\longrightarrow \B$ such that $F'P=F$, where
$P:\C\longrightarrow\C/\aut_1 (F)$ is the Galois covering given by the categorical quotient.
\end{thm}

Examples of coverings that are not Galois can be found in
\cite{CRS, le,le2}.

\begin{defi}
A \emph{universal covering} $U:\U \to \B$ is a Galois covering of $\B$ such that for any Galois covering $F:\C \to \B$,
and for any $u\in \U_0$, $c\in \C_0$ with $U(u)=F(c)$, there exists a unique morphism $(H,1)$ from $U$ to $F$ such that $H(u)=c$.
\end{defi}

In case of existence, a universal covering is unique up to isomorphisms of
Galois coverings. In general universal coverings do not exist, as the following Example shows.

\begin{exa}\label{ejemplo}
Let $k$ be a field and consider the $k$-categories

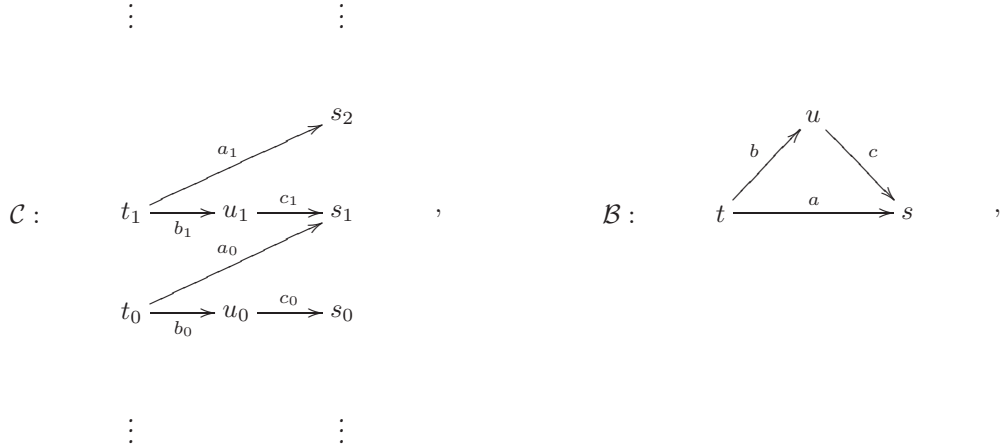
\[ \xymatrix{
&\vdots & & \vdots   & & & \\
&& & s_2  & & &  && u   \ar[dr]^c &\\
\C :  & t_1 \ar[urr]^{a_1} \ar[r]_{b_1} & u_1 \ar[r]^{c_1} & s_1 & , & &  \B :  & t \ar[rr]^{a} \ar[ur]^{b} &  & s & ,\\
&t_0 \ar[urr]^{a_0} \ar[r]_{b_0} & u_0 \ar[r]^{c_0} & s_0 & & & \\
& \vdots & & \vdots & & &
}\]
\medskip
and $F_1, F_2: \C \to \B$ Galois coverings of $\B$ given by
$F_1(a_i)=a$, $F_2(a_i)= a + cb$, $F_1(b_i)= F_2(b_i)=b$ and $F_1(c_i)=F_2(c_i)=c$.

Since $\C$ is simply connected, it admits no proper Galois covering , see \cite{AS, MP}, and there is no morphism between $F_1$ and $F_2$. Hence $\B$ admits no universal covering. \end{exa}

%%%%%%%%%%%%%%%%%%%%%%%%%%%%%%%%%%%%%%%%%%%%%%%%%%%%%%%%%%%%%%%%%%%%%%%%%%%%%%

\section{\sf Fibre product of coverings}

The fibre product of coverings that we define below will be useful in order to provide a
criterion for a covering to be Galois and to be universal.

\begin{defi}
Let $F:\C\longrightarrow\B$ and $G:\D\longrightarrow\B$ be $k$-functors of
$k$-categories. The \emph{fibre product} $\C\times_{\B}\D$ is the $k$-category defined as follows:
objects are pairs $(x,y)\in\C_0\times\D_0$ such that $F(x)=G(y)$; the set of morphisms from
$(x,y)$ to $(x',y')$ is the $k$-submodule of ${}_{x'}\C_x\oplus {}_{y'}\D_y$ given by
pairs of morphisms $(\varphi,\psi)$ verifying $F\varphi=G\psi$. Composition of morphisms
is defined componentwise.
\end{defi}

\begin{rem}
The fibre product as defined above is in fact the categorical fibre product in
the category of $k$-linear categories.  In particular, the following diagram is
commutative

\[ \xymatrix{
\C\times_{\B}\D \ar[r]^{{\sf pr}_{\C}} \ar[d]_{{\sf pr}_{\D}}  &  \C \ar[d]^F \\
\D  \ar[r]^G & \B }\]

\end{rem}

\begin{thm}\label{fullyfaithful}
With the same notations as above, assume that $F$ is a covering and that $G$ is fully
faithful. Then $\mathsf{pr}_\D$ is a covering of $\D$.
\end{thm}

\begin{proof}
Let $d$ be an object of $\D$. The $\mathsf{pr}_\D$-fibre of $d$ is the set of pairs
$(c, d)$ such that $Fc=Gd$, that is, the set of pairs $(c, d)$
such that $c$ is in the $F$-fibre of $Gd$. \\
In order to prove that $\left({\mathsf{pr}_\D}\right)_{d}^{(c, d)}$ is an
isomorphism between stars, we first show that it is surjective. Let $f\in
{}_{y}\D_{d}$, consider $Gf\in\st_{Gd}\B$ and let $\sum g_i$ be the unique sum of
morphisms in $\C$ obtained through the isomorphism of stars between $\st_{c}\C$ and
$\st_{Fc}\B = \st_{Gd}\B $. Of course each $g_i$ starts at the common object $c$ and it ends at
objects in the $F$-fibre of $Gy$. We have that $\sum Fg_i = Gf$. Since $Fg_i$ is a
morphism from $Gd$ to $Gy$ in $\B$ and $G$ is full, there is one morphism $h_i$ from
$d$ to $y$ in $\D$ such that $Gh_i=Fg_i$. This provides morphisms $(h_i,g_i)$ of the
fibre product. Moreover $G$ is faithful, $G\sum h_i=\sum Fg_i=Gf$ hence $\sum h_i = f$
and $\sum(g_i,h_i)$ is the required morphism in
the star $\st_{(c,d)} \C \times _{\B} \D$ lying above $f$.\\
In order to prove that $\left({\mathsf{pr}_\D}\right)_{d}^{(c, d)}$ is a
monomorphism, let $\sum (g_i,h_i) \in \st_{(c,d)} \C \times _{\B} \D$ be such that
$$0= \left({\mathsf{pr}_\D}\right)_{d}^{(c, d)}\sum (g_i,h_i) = \sum h_i \in \bigoplus_{i } {}_{y_i}\D_{d}.$$ Fix $i_0$ in the
set of indices of the last direct sum and consider the subset of indices defined by $J=\{
i : y_i=y_{i_0}\}$. Now $0= \sum_J h_i \in {}_{y_{i_{0}}}\D_{d}$ and $F (\sum_J g_i) =
\sum_J Gh_i=0$. Since $F_{F(c)}^{c} : \st_{c} \C  \to \st_{Fc} \B$ is an
isomorphism, we have that $\sum_J g_i=0$. Observe that each $g_i$ lies in a different
direct summand of $\st_{c} \C$.  Hence $g_i=0$ for all $i \in J$, $0=Fg_i=Gh_i$ and we
have that $h_i=0$ for all $i \in J$ because $G$ is faithful. Finally note that the
subsets $J=J(i_0)$ are a partition of the set of indices considered.
\end{proof}

We provide an example of a fibre product of coverings such that its projections are not coverings.

\begin{exa}\label{ejemplo2}
Consider coverings $F_1$ and $F_2$ as in Example \ref{ejemplo}.
Its fibre product $\C \times_\B \C$ has set
of objects isomorphic to the set $\C_0 \times \mathbb{Z}$ by the bijection sending $(x_i, x_{i+k})$ to $(x_i, k)$ for $x=s,t,u$.
Now the vector spaces of morphisms from $(t_i, t_j)$ to $(s_{i+1}, s_{j+1})$ are zero.
So, even if both projections are surjective on objects,
the respective maps between stars are not surjective.
\end{exa}

Next  we prove several results concerning fibre products of Galois coverings.
The first one gives a method to produce Galois coverings in an iterative way.
After recalling the definition of a trivial Galois covering we give a criterion useful to check if a given covering
is trivial.

\begin{pro}\label{square} Let $F:\C\longrightarrow\B$ be a Galois covering
of $k$-linear categories. The fibre product of $F$ by itself is a covering of $\C$,
through the projection functor $\mathsf{pr}_\C$.
\end{pro}
\begin{proof} It is clear that  ${\mathsf{pr}_\C}$ is surjective on objects since
${\mathsf{pr}_\C}(c,c)=c$. We must prove that
$$\left({\mathsf{pr}_\C}\right)_{c}^{(c,c')}: \st_{(c,c')} \C \times _{\B} \C \to \st_{c} \C$$
is an isomorphism, where $F(c)=F(c')$.  Using that $F$ is Galois we deduce that there
exists $g \in Aut(F)$ such that $gc=c'$.  The surjectivity on morphisms is also immediate
since $(f,gf) \in \st_{(c,c')} \C \times _{\B} \C $ and ${\mathsf{pr}_\C}(f,gf)=f$. The
proof of the injectivity is the same as in the proof of Theorem \ref{fullyfaithful}.
\end{proof}

\begin{defi}
A \emph{trivial covering} of a $k$-category $\B$ is a covering which is isomorphic to the
one given by the product $\B\times E$ of $\B$ by a set $E$, with objects $\B_0\times E$
and where the morphisms are ${}_{(y,e)}(\B\times E)_{(x,e)}= {}_y\B_x$ while
${}_{(y,f)}(\B\times E)_{(x,e)} =0$ if $e\neq f$. The covering functor is the projection
functor to the first factor.
\end{defi}

\begin{lem}
Let $F:\C \longrightarrow \B$ be a covering of $k$-categories, where $\B$ is connected.
Let $S$ be a section of $F$, namely a $k$-functor $\B\longrightarrow \C$ such that
$FS=1_\B$. Then the subcategory $S\B$ is full and it is a connected component of $\C$.
\end{lem}
\begin{proof} Let $f$ be a non-zero morphism of $\C$ having one
extreme object in $S\B$. Let $x \in S\B_0$ and let $f \in {}_y\C_{x}$. Let $b \in
\B_0$ be such that $x =Sb$. In order to prove that $y$ is also in the image of $S$,
consider $\st_{x}\C$ which is isomorphic through $F$ to $\st_{b}\B$. Since $FS=Id$,
we have that $S:\st_{b}\B\longrightarrow \st_{x}\C$ is the inverse of $F$ at the star
level. Hence $f=SFf$ and $y=SFy$. We have shown that a morphism between objects in the
image of $S$ is in the image of $S$, which shows that $S\B$ is full in $\C$. Since $S\B$
is isomorphic to $\B$, the category $S\B$ is connected.
\end{proof}

\begin{pro}\label{enough}
Let $F:\C \longrightarrow \B$ be a covering of $k$-categories, where $\B$ is connected.
The covering is trivial if and only if for each object $x\in\C_0$ there exists a
$k$-linear functor $S : \B\longrightarrow \C$ such that $FS$ is the identity functor of
$\B$ and $SFx=x$.
\end{pro}
\begin{proof} Assume that for each $x\in\C_0$ there exists a $k$-linear functor $S :
\B\longrightarrow \C$ such that $FS$ is the identity functor of $\B$ and $SFx=x$.  Let
$\Sigma$ be the set of all these sections of $F$ and consider the functor $\alpha: \B\times
\Sigma\longrightarrow \C$ given by $\alpha(b,S)=Sb$ and
$\alpha\left({}_{(c,S)}f_{(b,S)}\right)=Sf$. This functor is clearly faithful and dense.
It is also full by the previous lemma. The other implication is immediate.
\end{proof}

We already know after Proposition \ref{square} that when $F$ is a Galois covering
the fibre product of $F$ with itself is a covering. In fact, we will now characterize
Galois coverings using fibre products.

\begin{thm}\label{importante1}
A connected covering $F: \C\longrightarrow\B$ is Galois if and only the fibre product of
$F$ with itself is a trivial covering of $\C$ with respect to the projection functor
$\pr_{\C}$ to the first factor.
\end{thm}
\begin{proof} In case $F$ is Galois, let $x\in\C_0$ and let $(x,x')$ be in the
fibre of $x$ in the fibre product, namely $Fx=Fx'$. By transitivity and freeness of the action, there exists
a unique $g\in\aut (F)$ such that $gx=x'$. Let $S: \C \to \C \times_\B \C$ be the functor defined as
follows:
$Sy=(y,gy)$ on objects and $Sf=(f,gf)$ on morphisms. This functor verifies the hypotheses of the previous proposition, and hence $F$ is trivial.\\
Conversely let $x$ and $x'$ be in the same $F$-fibre, in order to define an automorphism
carrying $x$ to $x'$, let $S$ be the section of the functor of $\pr_1$ through $(x,x')$
obtained using Proposition \ref{enough}. The second component of $S$ will provide the
required automorphism $g$, more precisely $g$ is defined by $Sy=(y,gy)$ and $Sf=(f,gf)$.
Clearly $Fg=F$ by definition of the fibre product. Note also that this proves the
existence of $g$, we already know that there exists at most one endomorphism $g$ of $F$
such that $gx=x'$.  In order to prove that $g$ is invertible, note that for $x'$ and
$x''$ in the same fibre, there exists an endomorphism $h$ of $F$ such that $x'=hx''$. If
$x''=x$, we have $ghx=x$, then $gh=1$.
\end{proof}

We already know that since a universal covering $U$ of $\B$ is Galois, the fibre product
$\U \times_{\B} \U$ is trivial. In fact, we shall see that this is the case for
$\U \times_{\B} \C$, where $\C \to \B$ is any Galois covering of $\B$, and moreover,
this property characterizes universal coverings.

\begin{thm}\label{importante2}
A connected covering $U:\U\longrightarrow\B$ is
universal if and only if the fibre product of $U$ with any Galois covering
$F:\C\longrightarrow\B$ provides a trivial covering of $\U$.
\end{thm}
\begin{proof} In case $U$ is universal, let $(u,x)$ be an object
in $U\times_{\B} \C$, that is $Uu=Fx$. We have to prove that $\pr_{\U}$ is a covering and
that it is trivial. It is clear that  $\pr_{\U}$ is surjective on objects. To see that
$$\left({\pr_{\U}}\right)_u^{(u,x)}: \st_{(u,x)} \U \times_\B \C \to \st_u \U$$
is an isomorphism consider the $k$-linear functor $S: \U \to \U \times_\B \C$ defined as
follows: since $U$ is universal there exists a unique functor $G$ such that $FG=U$ and
$Gu=x$; take $S = Id_\U \times G$. Now this functor $S$ induces an inverse of
$\left({\pr_{\U}}\right)_u^{(u,x)}$.
It is also clear that $S$ is a section of $\pr_{\U}$ through $(u,x)$.\\
Conversely, let $u\in\U_0$ and $x\in\C_0$ such that $Uu=Fx$. In order to define $G$ such
that $FG=U$ and $Gu=x$, let $S$ be a section of $\pr_{\U}$ through $(u,x)$, and define
$G=\pr{_\C}S$. We have $FG=F\pr_{\C}S=U\pr_{U}S=U$.
\end{proof}

As an application we use this criterion in order to prove that the coverings in Example \ref{ejemplo} are not universal.

\begin{exa}
Consider the Galois coverings $F_1, F_2: \C \to \B$ as in Example \ref{ejemplo}.
We have already verified in Example \ref{ejemplo2} the projections from the fibre product $\C \times_\B \C$
to both coordinates are not coverings.
Hence, neither $F_1$ nor $F_2$ is a universal covering.
\end{exa}

%%%%%%%%%%%%%%%%%%%%%%%%%%%%%%%%%%%%%%%%%%%%%%%%%%%%%%%%%%%%%%%%%%%%%%%%%%%%%%%%%%%%%%%%%%%

\footnotesize \noindent Claude Cibils:\\Institut de Math\'{e}matiques et de Mod\'{e}lisation de Montpellier I3M, UMR 5149\\
Universit\'{e}  Montpellier 2, \\F-34095 Montpellier cedex 5,
France.\\
{\tt Claude.Cibils@math.univ-montp2.fr}\\

\noindent Mar\'\i a Julia Redondo:
\\Departamento de Matem\'atica,
Universidad Nacional del Sur,\\Av. Alem 1253\\8000 Bah\'\i a Blanca,
Argentina.\\ {\tt mredondo@criba.edu.ar} \\

\noindent Andrea Solotar:
\\Departamento de Matem\'atica,
 Facultad de Ciencias Exactas y Naturales,\\
 Universidad de Buenos Aires,
\\Ciudad Universitaria, Pabell\'on 1\\
1428, Buenos Aires, Argentina. \\{\tt asolotar@dm.uba.ar}

\end{document}